\documentclass[a4paper,12pt]{amsart}
\usepackage{amsmath,amsthm,amssymb,amsfonts,enumerate,color,hyperref,MnSymbol,wasysym}
\usepackage[pdftex]{graphicx}
\usepackage[bottom=2.5cm, top=2cm]{geometry}
\usepackage{float}
\usepackage{tikz}
\usepackage[makeroom]{cancel}

\theoremstyle{plain}
\newtheorem{theorem}{Theorem}[section]
\newtheorem{lemma}[theorem]{Lemma}
\newtheorem{corollary}[theorem]{Corollary}

\theoremstyle{definition}
\newtheorem{remark}[theorem]{Remark}

\newtheorem{thm}{Theorem} 

\numberwithin{equation}{section}

\def\supp{\operatorname{supp}}

\def\sgn{\operatorname{sgn}}

\def\Re{\operatorname{Re}}
\def\Im{\operatorname{Im}}
\def\im{\operatorname{Im}\left(\frac{1}{\Phi'}\right)}

\def\re{\operatorname{Re}\left(\frac{1}{\Phi'}\right)}

\def\imm{\operatorname{Im}\left({\Phi'}\right)}

\def\T{{T}_\Phi}
\def\A{\operatorname{Arg}}
\def\osc{\operatorname{osc}}
\def\hs{\operatorname{HS}}

\def\R{\mathbb R}

\def\C{\mathbb C}
\def\D{\mathcal{D}}

\def\e{\varepsilon}

\RequirePackage{fix-cm}

\usepackage{times}
\usepackage{comment}
\usepackage[latin1]{inputenc}
\usepackage[greek,english]{babel}
\usepackage{enumerate}
\usepackage{multicol,stackrel}
\usepackage{geometry}
\usepackage[all]{xy}
\usepackage{tikz}
\usepackage{pgfplots}
\pgfplotsset{compat=1.11}
\usepgfplotslibrary{fillbetween}
\usetikzlibrary{patterns}

\newcommand{\ntm}{\mathcal M} 

\newcommand{\abs}[1]{\left|#1\right|}

\begin{document}

\title{Rellich identities for the Hilbert transform} 
\author {Mar\'{\i}a J.\ Carro, Virginia Naibo and Mar\'{\i}a Soria-Carro} 

\address{Mar\'{\i}a Jes\'us Carro.  Department of  Analysis and Applied Mathematics,              
Universidad Complutense de Madrid, Plaza de las Ciencias 3,
28040 Madrid, Spain. Instituto de Ciencias Matem\'aticas ICMAT, Madrid, Spain}
\email{mjcarro@ucm.es}

\address{Virginia Naibo. Department of Mathematics, Kansas State University.
138 Cardwell Hall, 1228 N. Martin Luther King Jr. Dr., KS  66506, USA.}
\email{vnaibo@ksu.edu}

\address{Mar\'{\i}a Soria-Carro. Department of Mathematics, Rutgers University, 110 Frelinghuysen Rd.
Piscataway, NJ 08854, USA.}
\email{maria.soriacarro@rutgers.edu}


\thanks{First author partially supported by grants PID2020-113048GB-I00 funded by MCIN/AEI/10.13039/501100011033,  CEX2019-000904-S funded by MCIN/AEI/ 10.13039/501100011033 and Grupo UCM-970966 (Spain). Second author partially supported by the NSF under grant DMS 2154113 and the Simons Foundation under grant 705953 (USA)}

\begin{abstract}  We prove  Hilbert transform identities involving conformal maps via the use of Rellich identity and the solution of  the Neumann problem in a graph Lipschitz domain in the plane. We obtain as consequences new $L^2$-weighted estimates for the Hilbert transform, including a sharp bound for its norm as a bounded operator in weighted $L^2$ in terms of a weight constant associated to the Helson-Szeg\"o theorem.

\end{abstract}

\maketitle
\pagestyle{headings}\pagenumbering{arabic}\thispagestyle{plain}

\section{Introduction and main results}\label{sec:intro}

Let $H$ be the Hilbert transform, which is defined as
$$
Hf(x)= \frac{1}{\pi}\lim_{\varepsilon\to 0} \int_{|x-y|>\varepsilon} \frac{f(y)}{x-y} \, dy;
$$
 Hunt--Muckenhoupt--Wheeden~\cite{MR312139} proved that if $1<p<\infty,$ then $H$ is bounded from $L^p(w)$ to $L^p(w)$ if and only if $w\in A_p.$ Here $L^p(w)$ is the Lebesgue space of measurable  functions defined on $\R$ that are $p$-integrable with respect to the measure $w(x) dx$ and $A_p$ denotes the Muckenhoupt class of weights on $\R.$

Applications in Partial Differential Equations (see Fefferman--Kenig--Pipher~\cite{MR1114608}) led to a profound study of 
 sharp bounds for the operator norm of the Hilbert transform in terms of  the $A_p$ constant of the weight, denoted $[w]_{A_p}$ (see Section~\ref{sec:weights} for its definition). In particular, the $A_2$-conjecture for the Hilbert transform consisted in proving that the Hilbert transform satisfies analogous estimates to those by the Hardy-Littlewood maximal operator, that is, 
\begin{equation}\label{eq:A2conj}
||Hf||_{L^2(w)} \lesssim [w]_{A_2} ||f||_{L^2(w)} \quad \forall f\in L^2(w).
\end{equation}
 This conjecture was solved by Petermichl~\cite{MR2354322} in 2007 and numerous other related results followed concerning its   extension to singular integral operators and other classical operators, including Petermichl~\cite{MR2367098}, Hyt\"onen~\cite{MR2912709}, Hyt\"onen et al.~\cite{MR3176607,MR2993026}, Lerner~\cite{MR3085756} and Cruz-Uribe--Martell--P\'erez~\cite{MR2854179}. 

General  necessary and sufficient conditions on the weights for the boundedness of the Hilbert transform on weighted $L^2$ spaces were  first obtained by Helson-Szeg\"o~\cite{MR121608} in 1960  using complex variable techniques. More precisely, they proved that 
\begin{equation*}
H:L^2(w) \longrightarrow L^2(w) \quad\iff\quad w=e^{f_1+K f_2}, \ f_1,f_2\in L^\infty(\R), \ ||f_2||_{L^\infty} <\pi/2, 
\end{equation*}
where $K$ is a version of the Hilbert transform for $L^\infty$-functions (see Section~\ref{sec:weights}). As a consequence, we have
\begin{equation}\label{eq:hs}
w\in A_2  \quad\iff\quad w=e^{f_1+K f_2}, \quad f_1,f_2\in L^\infty(\R), \ ||f_2||_{L^\infty} <\pi/2.
\end{equation}
It is worth mentioning  that the implication to the left is easy to prove, but no direct proof of the implication to the right is presently known; the reader is directed to the work Garc\'{\i}a-Cuerva~\cite{MR783598} for an interesting survey on the topic. In particular, it is proved in \cite{MR783598}  that, if $||f_2||_{L^\infty} <\pi/2$, then
\begin{equation}\label{eq:A2cos}
[e^{Kf_2}]_{A_2} \le   \sec^2 ||f_2||_{L^\infty}.
\end{equation}

Given $w\in A_2,$  define the \textit{Helson-Szeg\"o constant} of $w,$ $[w]_{A_2(\hs)},$ as
$$
[w]_{A_2(\hs)}^2:=  \inf_ {\D_{w}} e^{\osc f_1} \sec^2\|f_2\|_{L^\infty},
$$
where  $\osc f_1 = \sup f_1-\inf f_1$ and   
$$
\D_w := \big\{ (f_1, f_2) :  w= e^{f_1 + Kf_2} \text{ with } f_1 ,f_2 \in L^\infty(\R) \ \text{ and }~ \|f_2\|_{L^\infty} <\pi/2\big\}.
$$
The inequality $
[w]_{A_2}\le  [w]^2_{A_2(\hs)}$ follows from \eqref{eq:A2cos} and, along with \eqref{eq:A2conj},   leads to 
\begin{equation*}
||Hf||_{L^2(w)} \lesssim [w]^2_{A_2(\hs)} ||f||_{L^2(w)}.
\end{equation*}
One of the results in this article improves the dependence on the Helson-Szeg\"o constant of the weight in the above inequality; indeed, we show that the dependence is linear in $[w]_{A_2(\hs)}$. More precisely, we have the following estimate:
\begin{theorem}\label{thm:A2est}  For every $w\in A_2$ and $f\in L^2(w)$ real-valued, it holds that
\begin{equation}\label{eq:A2est}
||Hf||_{L^2(w)} \lesssim [w]_{A_2(\hs)} ||f||_{L^2(w)}. 
\end{equation}
The dependence of $[w]_{A_2(\hs)}$ in \eqref{eq:A2est} is sharp.
\end{theorem}
The estimate \eqref{eq:A2est} is sharp in the sense that the norm of the operator $H$ as a bounded operator on $L^2(w)$ is comparable to $[w]_{A_2(\hs)}$ for some $w\in A_2.$
We hope that \eqref{eq:A2est} may lead to a new proof of \eqref{eq:A2conj}  by showing that $[w]_{A_2(\hs)}\lesssim [w]_{A_2}  $, which at the moment is an open question. 

The estimate \eqref{eq:A2est} along with other new $L^2$-weighted estimates for the Hilbert transform proved in this article are consequences of our main result on Rellich-type identities for the Hilbert transform. Such identities involve a conformal map $\Phi$ such that $\Omega=\Phi(\R^2_+)$ is a graph Lipschitz domain, that is, the upper part of the graph of a real-valued  Lipschitz function; the map  $\Phi$ extends as a homeomorphism from $\overline{\R^2_+}$ onto $\overline{\Omega}$ and  $\Phi'(x)$ exists and is non-zero for almost every $x\in \R$ (see details in Section~\ref{sec:cm}). Our main result is the following theorem.

\begin{theorem} \label{thm:rellich}
Let $\Phi$ be a conformal map as described in Section~\ref{sec:cm} and
 $f\in L^2( |\Phi'|^{-1})$ be real-valued. Then
\begin{equation}\label{eq:rellich}
\int_\mathbb R (Hf)^2 \re dx = \int_\mathbb R f^2 \re dx - 2 \int_\mathbb R f \,Hf \im dx,
\end{equation}
and
\begin{equation}\label{eq:rellichbis}
\int_\mathbb R (Hf)^2 \im dx = \int_\mathbb R f^2 \im dx + 2 \int_\mathbb R f \,Hf \re dx. 
\end{equation}
\end{theorem}

We remark that there is a strong connection between the Helson-Szeg\"o result \eqref{eq:hs} and conformal maps $\Phi$ as in the statement of Theorem~\ref{thm:rellich}. As proved in Kenig~\cite[Theorem 1.10 and Lemma 1.11]{MR556889}, it holds that
\begin{equation*}
|\Phi'|\in A_2
\end{equation*}
 and a converse result is true in the sense that if $w\in A_2$ then there exists $\Phi$ such that $\abs{\Phi'}\sim w.$ Indeed, the latter is a consequence of \eqref{eq:hs}: Let   $w=e^{f_1+Kf_2}$ with $(f_1,f_2)\in D_w$, then it is proved in \cite{MR556889} that there exists a conformal map $\Phi$ from $\R^2_+$ onto a graph Lipschitz domain so that 
\begin{equation}\label{eq:Phihs}
 \Phi'(x) = e^{Kf_2(x)} e^{-i f_2(x)} \quad \text{a.e. } x\in \R;  
 \end{equation}
 therefore, $\abs{\Phi'}=e^{-f_1}w$ and it follows that $\abs{\Phi'}\sim w.$ 
 
The main ingredients  in the proof  of Theorem~\ref{thm:rellich} are the following tools: (a)  the theory of solutions of  the Neumann problem in a graph Lipschitz domain with data in $L^2$ through the use of conformal maps as developed in Carro--Naibo--Ortiz~\cite{MR4542711} and, (b) Rellich's identity, which gives that if $\Phi$ is as in the statement of Theorem~\ref{thm:rellich}, then for every harmonic function $u$ in $\Omega=\Phi(\R^2_+)$ so that $M(\nabla u)\in L^2(\partial\Omega)$ and for every constant vector $e\in \R^2$, 
\begin{equation}\label{eq:rellichid}
\int_{\partial \Omega} |\nabla u|^2 \,(e\cdot \nu) \,d\sigma = 2 \int_{\partial \Omega} (\partial_\nu u) \, (e\cdot \nabla u) \,d\sigma,
\end{equation}
where $M$ is the non-tangential maximal operator and $d\sigma$ denotes integration with respect to arc-length.
The integral identity \eqref{eq:rellichid} is due to Rellich~\cite{MR2456} (see also Escauriaza--Mitrea~\cite[(2.35)]{MR2091359}); this identity and related versions play fundamental roles in questions on elliptic partial differential equations, inverse problems,  acoustic scattering, and the multiplier method; see  Agrawal--Alazard~\cite{AA2022} and references therein.

We note that the proof of \cite[Lemma 10]{MR312139} shows that for an infinitely differentiable function $f$ with compact support in $\R,$ it holds that
\begin{equation}\label{eq:HMW}
\int_{-\infty}^\infty (f+iHf)^2 \Phi'\,dx=0,
\end{equation}
where $\Phi$ is a conformal map  from $\R^2_+$ onto a graph Lipschitz domain such that $\Phi'=e^{K g} e^{-ig}$ for some $g$ satisfying  $\|g\|_{L^\infty}<\frac{\pi}{2}.$ Instances of \eqref{eq:rellich} and \eqref{eq:rellichbis} with $\Phi'$ instead of $1/\Phi'$ can then be deduced by taking the real and imaginary parts of \eqref{eq:HMW}. However, \eqref{eq:rellich} and \eqref{eq:rellichbis}  are not explicitly shown in \cite{MR312139} and the proof of \eqref{eq:HMW} is based on complex variable techniques that are different from the novel approach we use in the proof of Theorem~\ref{thm:rellich}, which holds for more general conformal maps $\Phi$.

The paper is organized as follows.  In Section \ref{sec:prelim}, we present preliminaries regarding weights, the Hilbert transform, conformal mappings and the solution of the Neumann problem in a graph Lipschitz domain in the plane with data in $L^2.$ In Section~\ref{sec:rellichL2}, we prove the Rellich's identity versions for the  Hilbert transform presented in Theorem~\ref{thm:rellich} as well as $L^2$-weighted estimates for $H$ that follow from them, including \eqref{eq:A2est}.
 Other applications of Theorem~\ref{thm:rellich} concerning Hilbert transform identities with power weights are discussed in Section~\ref{sec:power}. Finally, we present in Section~\ref{sec:rellichLp} versions of Rellich identities for the Hilbert transform in weighted $L^p$ spaces.

\section{Preliminaries}\label{sec:prelim}

In this section we present preliminaries regarding weights, the Hilbert transform, conformal mappings and the solution of the Neumann problem in a graph Lipschitz domain in the plane with data in $L^2.$

\subsection{Muckenhoupt weights and the Hilbert transform for $L^\infty$ functions}\label{sec:weights}

Let $w$ be a weight on $\R,$ i.e. a non-negative locally integrable function defined in $\R,$ and $1\le p\le\infty.$ We will denote by $L^p(w)$ the space   of measurable functions  $f:\R\to\mathbb{C}$ such that 
\begin{equation*}
 \|f\|_{L^p(w)}=\left(\int_\R |f(x)|^pw(x)\,dx\right)^{\frac{1}{p}}<\infty,
 \end{equation*}
 with the corresponding changes for $p=\infty.$ If  $w\equiv 1,$  we will use the notation  $L^p(\R)$ instead of  $L^p(w).$ 
 
 For $1<p<\infty,$ a weight $w$ defined on $\R$ belongs to the Muckenhoupt class $A_p$ if 
 \begin{equation}\label{eq:Ap}
[w]_{A_{p}}=\sup _{I\subset \R}\, \left(\frac{1}{|I|}\int_{I} w(x) d x\right)\left(\frac{1}{|I|}\int_{I} w(x)^{1-p^{\prime}}d x\right)^{p-1}<\infty,
\end{equation}
where the supremum is taken over all intervals contained in $\R.$ 
We recall that  $w\in A_p$ if and only if $w^{1-p'}\in A_{p'},$ where $p'$ is the conjugate exponent of $p$ (i.e. $1/p+1/p'=1$); also, $A_p\subset A_q$ if $p<q$ and for every $w\in A_p$ there exist $\varepsilon>0$ such that $w\in A_{p-\varepsilon}.$

We define the operator $K$ as
\begin{equation*}
Kf(x)=\frac{1}{\pi}\lim_{\varepsilon\to 0} \int_{\abs{x-y}>\varepsilon} f(y)\left(\frac{1}{x-y}+\frac{\chi_{\abs{y}>1}(y)}{y}\right)\,dy,
\end{equation*}
which allows to extend the definition of the Hilbert transform to $L^\infty(\R).$

For brevity we will use the notation $A\lesssim B$ to mean that $A\le c\, B$, where $c$ is a constant that may only depend on some of the parameters  but not on the functions or weights involved.

\subsection{Conformal maps and the Neumann problem in a graph Lipchitz domain in the plane}\label{sec:cm}
In this section we define the conformal maps that will be considered throughout this work and recall results from \cite{MR4542711} on the solution of the Neumann problem in a graph Lipschitz domain with data in $L^2$, which is obtained via the use of conformal maps and the solution of the Neumann problem in the upper-half plane.

 Let $\Lambda$  be a curve in the complex plane given parametrically by $\xi(x) = x+ i \gamma (x)$ for $x\in \R$, where $\gamma$ is a real-valued  Lipschitz function  with  constant $k,$  and  consider {\it the graph  Lipschitz domain} 
 \begin{equation}\label{eq:omega}
\Omega=\{ z_1+iz_2\in \mathbb{C}: z_2> \gamma(z_1) \}.
\end{equation}
Then $\Lambda=\partial\Omega$ and, since $\Omega$ is simply connected, there exists a conformal map   $\Phi: \mathbb R^2_+ \longrightarrow \Omega$  such that $\Phi(\infty) =\infty$ and $\Phi(i)=iy_0$ for some $y_0>\gamma(0).$ The map $\Phi$ extends as a homeomorphism from $\overline{\R^2_+}$ onto $\overline{\Omega}$ and $\Phi(x),$ $x\in \R,$ is absolutely continuous when restricted to any finite interval; this implies that $\Phi'(x)$ exists for almost every $x\in \R$ and is locally integrable. It also holds that $\Phi'(x)\ne 0$ for almost every $x\in\R,$    $\lim_{z\to x}\Phi'(z)=\Phi'(x)$ in the non-tangential sense for almost every $x\in \R$  and $|\Phi'|\in A_2.$ If $\Phi'(x)$ exists and is not zero, then it is a vector tangent to $\partial\Omega$ at $\Phi(x).$ See  Kenig~\cite[Theorems 1.1 and 1.10]{MR556889} for the proof of those properties. The inverse of $\Phi$ will be denoted $\Psi.$

Given a measurable function $g$ defined in $\partial\Omega,$ let $\T$  be given by
\begin{equation*} 
\T{g}(x)= |\Phi'(x)| g(\Phi(x)),\quad   x\in \R.
\end{equation*}
Denote by $L^2(\partial\Omega)$ the Lebesgue space of measurable functions defined on $\partial\Omega$ that are square-integrable with respect to arc-length. We note that   $\T$  is a bijection from $L^2(\partial\Omega)$  onto $L^2(|\Phi'|^{-1})$ and  we have $\|g\|_{L^2(\partial\Omega)}=\|\T{g}\|_{L^2(|\Phi'|^{-1})}.$

For $g\in L^2(\partial\Omega),$ consider the classical Neumann  problem in $\Omega:$
\begin{equation}\label{eq:neumann}
  \Delta v =0 \text{ on }\Omega, \qquad \partial_\nu v=g\text{ on } \partial\Omega \quad \text{and}\quad \ntm_\alpha (\nabla v)\in L^2(\partial\Omega).
  \end{equation}
  Here  $\Delta$ is the Laplace operator,  $\nu$ denotes the outward unit normal vector to $\partial\Omega,$  $ \partial_\nu v=\nabla v\cdot \nu$ and the equality  $\partial_\nu v=g$ is meant in the non-tangential convergence sense. For $0<\alpha<\arctan(1/k),$  $\ntm_\alpha$ denotes the non-tangential maximal operator   given by 
 \begin{equation*}
\ntm_\alpha (F)(\xi)= \sup_{z\in \Gamma_\alpha(\xi)} |F(z)|,\quad \xi\in \partial\Omega,
\end{equation*}
for a complex-valued function $F$ defined in $\Omega$ and 
$$\Gamma_\alpha(\xi)=\{z_1+i z_2\in \mathbb{C}: z_2>\text{Im}(\xi)  \text{ and } |\text{Re}(\xi)-z_1|<\tan(\alpha) |z_2-\text{Im}(\xi)|\}.$$ 

Inspired by tools and techniques  from Kenig~\cite{MR545265, MR556889}, it was proved in \cite[Theorem 1.4]{MR4542711} that  for every $g\in L^2(\partial\Omega),$  $v=u_{\T{g}}\circ \Psi$ is a solution of the Neumann problem \eqref{eq:neumann}  and
\begin{equation*}
\|\ntm_\alpha(\nabla v)\|_{L^2(\partial\Omega)} \lesssim \|g\|_{L^2(\partial\Omega)},
\end{equation*}
where, for $f:\R\to \mathbb{C},$ $u_f$ is defined by \begin{equation}\label{eq:solneumannR2}
u_f(x,y):=-\frac{1}{\pi}\int_{\R}  \log\left(\textstyle{\frac{\sqrt{(x-t)^2+y^2}}{1+|t|}}\right) f(t)\,dt,\quad (x,y)\in \R^2_+.
 \end{equation}

 We note that the integral on the right-hand side of \eqref{eq:solneumannR2} is absolutely convergent for all $f$ satisfying $\int_{\R}\frac{|f(t)|}{1+|t|}\,dt<\infty;$ in particular, it is well defined for any $f\in L^2(w)$ with $w\in A_2.$   As shown in \cite{MR4542711},  $u_f$ is a solution of the Neumann problem in the upper half plane: More precisely, if  $w\in A_2$ and $f\in L^2(w),$  then $u_f$ is harmonic in $\R_+^2,$ $\nabla u\cdot {(0,-1)}=f$ on $\R$ in the sense of non-tangential convergence and 
 \begin{equation*}
\|\ntm_\alpha(\nabla u_f)\|_{L^2(w)} \lesssim \|f\|_{L^2(w)}.
\end{equation*}  
 It follows  that 
\begin{align*}
\partial_x u_f(x, y)= -(Q_{y}*f)(x)\quad \text{ and } \quad 
\partial_y u_f(x, y)=-(P_{y} * f)(x), 
\end{align*}
where, for $y>0,$ $P_y $ is the Poisson kernel and $Q_y$ is the conjugate of the Poisson kernel. As a consequence, 
\begin{equation} \label{eq:ufderiv}
\partial_x u_f (x,0) = -H f (x) \quad \hbox{ and } \quad \partial_y u_f(x,0) = -f (x),
\end{equation}
 for almost every $x\in \R$ in the non-tangential convergence sense.
 
 \bigskip
 
 For easier referencing, we state as a theorem the following result mentioned in Section~\ref{sec:intro}. 
\begin{thm}[see proof of Lemma 1.11 in \cite{MR556889}]\label{thm:kenig}
 If $w\in A_2$ and $w=e^{f_1+Kf_2}$ with $(f_1,f_2)\in D_w,$ then there exists a conformal map $\Phi_w$ from $\R^2_+$ onto a graph Lipschitz domain so that 
\begin{equation}\label{eq:kenig}
 \Phi'_w(x) = e^{Kf_2(x)} e^{-i f_2(x)} \quad \text{a.e. } x\in \R.
 \end{equation}
 In particular $\abs{\Phi_w'}=e^{-f_1}w=e^{Kf_2},$ $\A \Phi'_w=-f_2$ and  $\abs{\Phi'_w}\sim w.$ 
 \end{thm}

 \section{Rellich's identities and new $L^2$-weighted estimates for $H$}\label{sec:rellichL2}
 
 In this section we prove the Rellich's identity versions for the  Hilbert transform stated in Theorem~\ref{thm:rellich} and new  $L^2$-weighted estimates for $H$ that follow from them. We start with some preliminaries in Section~\ref{sec:rellichL2:prem}, give the proof of Theorem~\ref{thm:rellich} in Section~\ref{sec:rellichL2:proof} and present the $L^2$-weighted estimates for $H$ in Section~\ref{sec:rellichL2:est}.

\subsection{Preliminaries}\label{sec:rellichL2:prem}
Let $\Omega,$  $\Phi$ and $\Psi$ be as in Section~\ref{sec:cm}. We can parametrize the curve $\partial\Omega$ using the conformal map  $\Phi$, that is,
$$
\partial\Omega = \{ z\in \C : z=\Phi(x) \ \text{ for some } x\in \R\}.
$$
We write $\Phi=\Phi_1+ i \Phi_2$, where $\Phi_1=\Re(\Phi)$ and $\Phi_2=\Im(\Phi),$ and analogously,  $\Psi=\Psi_1+i\Psi_2.$ For convenience, we will sometimes denote a complex number $z=z_1+i z_2$ using vector notation, i.e.,  $z=(z_1,z_2)$.
Let $\nu(z)$ be the outward unit normal vector to $\partial\Omega;$ then,
\begin{equation}\label{eq:normal}
\nu(z) = \frac{\big( \Phi_2'(\Psi(z)), - \Phi_1'(\Psi(z)) \big)}{|\Phi'(\Psi(z))|}. 
\end{equation}
Since  $\Psi$ is also a conformal map, it satisfies the Cauchy--Riemann equations:
\begin{equation} \label{eq:CR}
\partial_1 \Psi_1 = \partial_2 \Psi_2 \quad \text{and} \quad \partial_2 \Psi_1 = - \partial_1 \Psi_2,
\end{equation}
where $\partial_1$ and $\partial_2$ denote the partial derivatives with respect to the first and second variable, respectively. The following lemma will be useful in the proof of Theorem~\ref{thm:rellich}.

\begin{lemma}\label{lem:devpsi}
For almost every $x\in \R$, it holds that
\begin{align*}
(\partial_1 \Psi_1) (\Phi(x)) &= \frac{\Phi_1'(x)}{|\Phi'(x)|^2} = \Re\left(\frac{1}{\Phi'(x)}\right),\\
(\partial_2 \Psi_1) (\Phi(x)) &= \frac{\Phi_2'(x)}{|\Phi'(x)|^2}= - \Im\left(\frac{1}{\Phi'(x)}\right).
\end{align*}
\end{lemma}

\begin{proof}
Since $\Psi(\Phi(x,y))=(x,y),$ we have
$$\Psi_1(\Phi(x,y))=x, \quad \hbox{for all}~(x,y) \in \overline{\R^2_+}. $$
Differentiating with respect to $x$, we get
\begin{equation*}
(\partial_1 \Psi_1) (\Phi(x,y)) \partial_x \Phi_1(x,y) + (\partial_2 \Psi_1)(\Phi(x,y)) \partial_x \Phi_2(x,y)=1,
\end{equation*}
and similarly, differentiating with respect to $y$, we have
\begin{equation*}
(\partial_1 \Psi_1) (\Phi(x,y)) \partial_y \Phi_1(x,y) + (\partial_2 \Psi_1)(\Phi(x,y)) \partial_y \Phi_2(x,y)=0.
\end{equation*}
Furthermore, since $\Phi$ satisfies the Cauchy--Riemann equations, $\partial_x \Phi_1 = \partial_y \Phi_2$ and $\partial_y \Phi_1 = - \partial_x \Phi_2$. Set $y=0$, and denote $\Phi_1'(x)=\partial_x \Phi_1(x,0)$, $\Phi_2'(x)=\partial_x \Phi_2(x,0)$, and $\Phi'=\Phi_1'+i \Phi_2'$. Recall that $\Phi'(x)$ exists and is non-zero for almost every $x\in \R.$ Then
\begin{align} \label{eq:diffx}
(\partial_1 \Psi_1) (\Phi(x)) \Phi_1'(x) + (\partial_2 \Psi_1) (\Phi(x)) \Phi_2'(x)=1\\ \label{eq:diffy}
- (\partial_1 \Psi_1) (\Phi(x)) \Phi_2'(x) + (\partial_2 \Psi_1) (\Phi(x)) \Phi_1'(x)=0.
\end{align}
Multiplying \eqref{eq:diffx} by $\Phi_2'(x)$ and \eqref{eq:diffy} by $\Phi_1'(x)$, and adding both equations, we see that
$$
(\partial_2 \Psi_1)(\Phi(x))= \frac{\Phi_2'(x)}{|\Phi'(x)|^2}, \quad\hbox{for a.e.}~x\in \R.
$$
Finally, substituting this expression into \eqref{eq:diffy} when $\Phi'_2(x)\neq 0$, we get
$$
(\partial_1 \Psi_1)(\Phi(x))= \frac{\Phi_1'(x)}{|\Phi'(x)|^2}, \quad\hbox{for a.e.}~x\in \R.
$$
If $\Phi'_2(x)=0,$ then \eqref{eq:diffx} gives $(\partial_1 \Psi_1)(\Phi(x))=  \frac{1}{\Phi_1'(x)}=\frac{\Phi_1'(x)}{|\Phi'(x)|^2}.$
\end{proof}

\subsection{Proof of Theorem~\ref{thm:rellich}}\label{sec:rellichL2:proof}

 Let $f\in  L^2(|\Phi'|^{-1})$ be real-valued; since $\T: L^2(\partial\Omega) \to L^2(|\Phi'|^{-1})$ is invertible, there exists a unique $g\in L^2(\partial\Omega)$ such that $f=\T g.$ As explained in Section~\ref{sec:cm},  a solution $v$ of the Neumann problem \eqref{eq:neumann}  with datum $g$ can be represented by 
$v= u \circ \Psi,$ where $u=u_f$ as given  in \eqref{eq:solneumannR2} and \eqref{eq:ufderiv} holds. Then the solution $v$ satisfies  Rellich's identity \eqref{eq:rellichid}:
$$
\int_{\partial\Omega} |\nabla v|^2 \,(e\cdot\nu) \, d\sigma = 2 \int_{\partial\Omega} g \,\,(e\cdot \nabla v) \, d\sigma,
$$
for any constant vector $e=(e_1,e_2)$.

Our goal is to derive a Rellich's type identity for the function $u$. To this end, we need to compute the partial derivatives of $v$ in terms of $u$. Let $z\in \partial\Omega$, with $z=\Phi(x)$, for some $x\in \R$. Using  Lemma~\ref{lem:devpsi},  \eqref{eq:ufderiv} and \eqref{eq:CR}, we obtain
\begin{align} \label{eq:der1}
(\partial_1 v)(\Phi(x)) 
 &= -H(\T g)(x)  \Re\left(\frac{1}{\Phi'(x)}\right) -   \T g (x) \Im\left(\frac{1}{\Phi'(x)}\right),     \\ \label{eq:der2}
(\partial_2 v)(\Phi(x)) 
&=  H(\T g)(x) \Im\left(\frac{1}{\Phi'(x)}\right) - \T g (x)   \Re\left(\frac{1}{\Phi'(x)}\right) .
\end{align}
From \eqref{eq:der1} and \eqref{eq:der2}, it follows that
\begin{equation} \label{eq:mod}
|\nabla v (\Phi(x))|^2 = \big( (H( \T g)(x))^2 + (\T g(x))^2 \big) \frac{1}{|\Phi'(x)|^2}. 
\end{equation}
Making the change of variables $z= \Phi(x),$  and using \eqref{eq:mod} and \eqref{eq:normal}, we see that
\begin{align*}
\int_{\partial\Omega} |\nabla v|^2\,\, ( e \cdot \nu ) \, d\sigma 
&= \int_\R \big( (H(\T g))^2 + (\T g )^2 \big) \,\left((e_1,e_2)\cdot  \frac{( \Phi_2', - \Phi_1')}{|\Phi'|^2}\right) \, dx\\
&= -e_1  \int_\R  ( (H(\T g))^2 + (\T g)^2 \big) \im\, dx\\
&\quad  - e_2 \int_\R  ( (H(\T g))^2 + (\T g )^2 \big) \re\, dx.
\end{align*}
Similarly, we get
\begin{align*}
2 \int_{\partial\Omega} g \,\,(e\cdot \nabla v) \, d\sigma
&= -2 e_1 \int_\R \T g \, H(\T g) \re\, dx - 2 e_1 \int_\R (\T g)^2 \im \, dx\\
& \quad + 2e_2 \int_\R \T g \, H(\T g) \im \, dx - 2 e_2 \int_R (\T g)^2 \re \, dx.
\end{align*}
Therefore, Rellich's identity yields 
\begin{align*}
&2 e_1 \int_\R \T g \,H(\T g) \re\, dx - e_2 \int_\R  ( (H(\T g))^2 -  (\T g )^2 \big) \re\, dx\\
& \quad =  e_1  \int_\R  ( (H(\T g))^2 -  (\T g)^2 \big) \im\, dx + 2 e_2  \int_\R \T g \,H(\T g) \im\, dx.
\end{align*}
Recalling that $\T g=f,$  \eqref{eq:rellich} and \eqref{eq:rellichbis} immediately follow from the above identity, taking $e=(0,1)$ and $e=(1,0)$, respectively.
\qed

\begin{remark}
Observe that for $\Omega=\R^2_+$, we have $\Phi(x,y)=(x,y)$. If $y=0$, then $\Phi'(x)=1$,  and noting that $\int_{\R}fHfdx=0,$ \eqref{eq:rellich} recovers the well-known identity $\|Hf\|_{L^2}=\|f\|_{L^2}$ for the Hilbert transform in $L^2(\R)$.   
\end{remark}

\subsection{$L^2$-weighted estimates for the Hilbert transform}\label{sec:rellichL2:est}
 
 In this section we prove new $L^2$-weighted estimates for the Hilbert transform that are consequences of  Theorem~\ref{thm:rellich}. In what follows  $\Phi$ is a conformal map as  described in Section~\ref{sec:cm}.

We first state and prove Theorem~\ref{thm1}, which gives $L^2$-weighted estimates  for $H$ with constants in terms of $\varphi(\tau^2),$ where $\tau=\tan \|\A \Phi'\|_{L^\infty}$  and  
$$
\varphi(s) := \inf_{0<\varepsilon<1} \frac{\varepsilon+ s }{(1-\varepsilon)\varepsilon} = 1 + 2s + 2\sqrt{s^2+s},\quad s\geq 0. 
$$
We then prove Theorem~\ref{thm:A2est}, which follows from Theorem~\ref{thm1} and  gives  sharp $L^2$-weighted estimates for $H$ in terms of the Helson-Szeg\"o constant of the weight as discussed in Section~\ref{sec:intro}.

We end the section with Corollary~\ref{coro:Hnorm}, which gives an identity for the norm of $H$ as a bounded operator on $L^2(\re),$ and Corollary~\ref{coro:Hmono}, which gives a uniform bound for the norm of $H$ as a bounded operator on $L^2(|\Phi'|^{-1})$ when $\Phi(\R^2_+)$ is a monotone Lipschitz domain.

\begin{theorem} \label{thm1}
Let $\Phi$ be a conformal map as described in Section~\ref{sec:cm} and
 $\tau=\tan \|\A \Phi'\|_{L^\infty}$. Then for all $f\in L^2( |\Phi'|^{-1})$ real-valued, we have 
\begin{align}
\int_\mathbb R (Hf)^2 \re\, dx &\le  \varphi(\tau^2)  \int_\mathbb R f^2 \re \, dx,\label{eq:thm1_1}\\
\int_\R (Hf)^2 |\Phi'|^{-1} \, dx &\le (2+\sqrt{3}) \varphi(\tau^2)   \int_\R f^2 \,|\Phi'|^{-1} \, dx. \label{eq:thm1_2}
\end{align} 
\end{theorem}

\begin{proof}  \underline{Proof of \eqref{eq:thm1_1}:} We first estimate  the last term on the right-hand side of  \eqref{eq:rellich}. Since $\im=\re \tan(\A\frac{1}{\Phi'})=- \re \tan(\A\Phi')$ and $\re=\frac{\Re(\Phi')}{|\Phi'|^2}\ge 0,$ it follows that 
$$
\int_\R f  Hf  \left|\im\right|\,  dx \le \tau \int_ \R f   Hf \re\, dx. 
$$
Therefore, using the Cauchy-Schwarz inequality and \eqref{eq:rellich}, we get
$$
\int_\mathbb R (Hf)^2  \re\, dx \le  \left(1+ \frac{\tau^2}{\e} \right) \int_\mathbb R f^2 \re\, dx +  \e  \int_\R  (Hf)^2    \re\, dx.
$$
Equivalently,
$$
\int_\mathbb R (Hf)^2  \re\, dx \le   \frac{\e+\tau^2}{(1-\e)\e}  \int_\mathbb R f^2 \re\, dx, $$
and the result follows by taking the infimum in $0<\varepsilon<1$.
\medskip

\noindent

\underline{Proof of \eqref{eq:thm1_2}:} Using \eqref{eq:thm1_1} and again that $\im=- \re \tan(\A\Phi')$ and $\re\ge 0,$ we have
\begin{align*}
\displaystyle\int_\R (Hf)^2 |\Phi'|^{-1} \, dx 
 & \leq  \varphi(\tau^2)  \int_\mathbb R f^2 |\Phi'|^{-1} \, dx + \tau \int_\mathbb R (Hf)^2 \re\, dx.
\end{align*}
We next use \eqref{eq:rellich}  and the Cauchy-Schwarz inequality to control the second term on the right-hand side: 
\begin{align*}
 \tau \int_\mathbb R (Hf)^2 \re\, dx
 & \leq \tau  \int_\mathbb R f^2 \re\, dx + \frac{\tau}{\e}  \int_\mathbb R f^2 |\im| \, dx\\
 & \quad + \e \tau  \int_\mathbb R (Hf)^2 |\im| \, dx\\
 & \leq \left( \tau + \frac{\tau}{\e}\right)  \int_\mathbb R f^2 |\Phi'|^{-1}\, dx + \e \tau  \int_\mathbb R (Hf)^2 |\Phi'|^{-1} \, dx\\
 & \leq \left(\frac{1+\tau}{\e}\right)  \int_\mathbb R f^2 |\Phi'|^{-1}\, dx + \e \tau  \int_\mathbb R (Hf)^2 |\Phi'|^{-1} \, dx,
\end{align*}
where  $\e>0$ is such that $0<\e \tau<1.$  
Combining both estimates,  we see that
\begin{align*}
(1- \e \tau) \displaystyle\int_\R (Hf)^2 |\Phi'|^{-1} \, dx
\leq   \left( \varphi(\tau^2) + \frac{1+\tau}{\e}\right)  \int_\mathbb R f^2 |\Phi'|^{-1} \, dx.
\end{align*}
Setting $\delta= \e \tau,$ we obtain
$$
\int_\R (Hf)^2 |\Phi'|^{-1} \, dx \leq   \varphi(\tau^2)  \frac{\delta+ (1+\tau)\tau/ \varphi(\tau^2) }{(1-\delta)\delta}  \int_\mathbb R f^2 |\Phi'|^{-1} \, dx.
$$
Taking the infimum over all $0<\delta<1$ leads to 
$$
\int_\R (Hf)^2 |\Phi'|^{-1} \, dx \leq   \varphi(\tau^2) \varphi((1+\tau)\tau/\varphi(\tau^2))  \int_\mathbb R f^2 |\Phi'|^{-1} \, dx.
$$
Since  $(1+\tau)\tau/\varphi(\tau^2)\le 1/2,$ we have  $\varphi((1+\tau)\tau/\varphi(\tau^2))\le \varphi(1/2)=2+\sqrt{3};$ therefore, it follows that 
$$
\int_\R (Hf)^2 |\Phi'|^{-1} \, dx \leq  (2+\sqrt{3})  \varphi(\tau^2)  \int_\mathbb R f^2 |\Phi'|^{-1} \, dx,
$$
as desired.
\end{proof}

We next present the proof of Theorem~\ref{thm:A2est}.

\begin{proof}[Proof of Theorem~\ref{thm:A2est}]  Since $w\in A_2,$ we have  $w^{-1}= e^{-f_1 - Kf_2}$ with $(f_1, f_2)\in \D_{w}. $ Let $\Phi=\Phi_{w^{-1}}$ be as given in Theorem~\ref{thm:kenig}; then $|\Phi'|=  e^{-Kf_2}$  and $\A\Phi'=f_2.$

By \eqref{eq:thm1_2} and using that $\varphi(s) \sim 1+s$ for $s\ge 0,$ we have 
\begin{align*}
\int_\mathbb R (Hf)^2 w \,dx&= \int_\R (Hf)^2 e^{f_1 +Kf_2} \,dx \\
& \le e^{\sup f_1} \int_\mathbb R (Hf)^2 |\Phi'|^{-1} \,dx \\
& \lesssim e^{\sup f_1}\varphi(\tan^2\|f_2\|_{L^\infty})   \int_\R f^2  |\Phi'|^{-1} \, dx\\
& \lesssim e^{\osc f_1} (1+\tan^2\|f_2\|_{L^\infty}) \int_\R f^2   w  \,  dx.
\end{align*}
Taking infimum over all pairs $(f_1, f_2)\in \D_{w},$ we obtain \eqref{eq:A2est}. 

The fact that the dependence of $[w]_{A_2(\hs)}$ in \eqref{eq:A2est} is sharp follows from Remark~\ref{re:sharp}.
\end{proof}

The next corollary is a direct consequence of \eqref{eq:rellich}.

\begin{corollary}\label{coro:Hnorm} If  $\Phi$ is a conformal map as described in Section~\ref{sec:cm}, then
\begin{equation}\label{eq:Hnorm}
||H||^2_{L^2\left(\re\right)\to L^2\left(\re\right)}= \mathop{\sup_{||f||_{{L^2\left(\re\right)}=1}}}_{f\textnormal{ real-valued}}
\left| 1- 2\int_\mathbb R f Hf \im dx\right|
\end{equation}
\end{corollary}
It is worth noting that for the integrals on the right hand side of \eqref{eq:Hnorm}, only the values $Hf(x)$ for $x$ in the support of $f$ are needed.

\medskip

A {\it{monotone graph Lipschitz domain}}  is a graph Lipschitz domain $\Omega$ as described in Section~\ref{sec:cm}  for which the function $\gamma$ is monotone. Note that if $\Phi$ is a conformal map associated to $\Omega$ as given in Section~\ref{sec:cm}, then $\Omega $ is a monotone graph Lipschitz domain  if and only if $\imm\ge 0$ almost everywhere or $\imm\le 0$ almost everywhere. 
We have the following result for conformal maps associated to monotone graph Lipschitz domains:

\begin{corollary}\label{coro:Hmono} Let $\Phi$ be a conformal map as described in Section~\ref{sec:cm} such that $\Phi(\R^2_+)$ is a monotone graph Lipschitz domain. If $f\in L^2(|\Phi'|^{-1})$ is real-valued, then
 \begin{align}\label{eq:Hmono}
||Hf||_{L^2(|\Phi'|^{-1})}\le \sqrt{\sqrt{2}\,\left(1+2\sqrt{2}+2\sqrt{2+\sqrt{2}}\right)} \,||f||_{L^2(|\Phi'|^{-1})}.
\end{align}
\end{corollary}
\begin{proof}
Since $\Phi(\R^2_+)$ is a  monotone graph Lipschitz domain, then $\im\ge 0$ almost everywhere or $\im\le 0$ almost everywhere. 

Assume first that $\im\ge 0$ almost everywhere. Adding \eqref{eq:rellich} and \eqref{eq:rellichbis}, we have
\begin{align*}
\int_\mathbb R (Hf)^2 &|\Phi|^{-1}dx\le
\int_\mathbb R (Hf)^2 \left(\re+\im\right) dx\\&=\int_\mathbb R f^2 \left(\re+\im\right)  dx 
+2 \int_\mathbb R f \,Hf \left(\re-\im\right) dx\\
&\le \sqrt{2} \int_\mathbb R f^2 |\Phi'|^{-1} dx
+2 \sqrt{2}\int_\mathbb R |f \,Hf| |\Phi'|^{-1} dx.
\end{align*}
Letting $0<\varepsilon<1/\sqrt{2}$ and applying the Cauchy-Schwarz inequality, the second term in the last line is controlled by 
\begin{align*}
 2 \sqrt{2}  \int_\mathbb R |f \,Hf| |\Phi'|^{-1}dx
\le \frac{\sqrt{2}}{\varepsilon} \int_\mathbb R f^2 |\Phi'|^{-1}dx +\sqrt{2}\,\varepsilon \int_\mathbb R (Hf)^2 |\Phi'|^{-1}dx.
\end{align*}
Setting $\delta=\sqrt{2}\,\varepsilon,$ we then obtain 
\begin{align*}
\int_\mathbb R (Hf)^2 |\Phi|^{-1}dx\le \sqrt{2}\frac{1+\frac{1}{\varepsilon}}{1-\sqrt{2}\varepsilon} \int_\mathbb R f^2 |\Phi'|^{-1}dx= \sqrt{2}\frac{\delta+\sqrt{2}}{\delta(1-\delta)} \int_\mathbb R f^2 |\Phi'|^{-1}dx.
\end{align*}
Taking infimum over $0<\delta<1,$ it follows that
\begin{align*}
\int_\mathbb R (Hf)^2 |\Phi|^{-1}dx&\le \sqrt{2}\,\varphi(\sqrt{2}) \int_\mathbb R f^2 |\Phi'|^{-1}dx\\
&=\sqrt{2}\,\left(1+2\sqrt{2}+2\sqrt{2+\sqrt{2}}\right) \int_\mathbb R f^2 |\Phi'|^{-1}dx.
\end{align*}

The case $\im\le 0$ follows analogously by subtracting \eqref{eq:rellich} and \eqref{eq:rellichbis}.

\end{proof}

\section{Hilbert transform identities with power weights}\label{sec:power}

In this section we investigate further the identities \eqref{eq:rellich} and \eqref{eq:rellichbis} when $\Phi(\R^2_+)$ is a cone, and obtain Hilbert transform identities with power weights. We will consider two types of cones: Symmetric cones about the imaginary axis and monotone cones (i.e. cones that are monotone graph Lipschitz domains). 

\subsection{Symmetric cones}\label{sec:symcones} Let $\Omega$ be a cone with aperture $\alpha \pi$, with $\alpha\in (0,2)$, which is symmetric about the imaginary axis (see Figure~\ref{fig1}). Consider the conformal map $\Phi : \R^2_+ \to \Omega$ such that
\begin{equation}\label{map1}
\Phi(z) = e^{i \tfrac{(1-\alpha)}{2} \pi} z^\alpha = i e^{-i \tfrac{\alpha}{2} \pi} e^{\alpha ( \log |z| + i \A(z) )},
\end{equation}
where we chose the branch cut $\{ i y : y \leq 0\}$, so that $\Phi$ is analytic on $\R^2_+$. 

\bigskip

\medskip

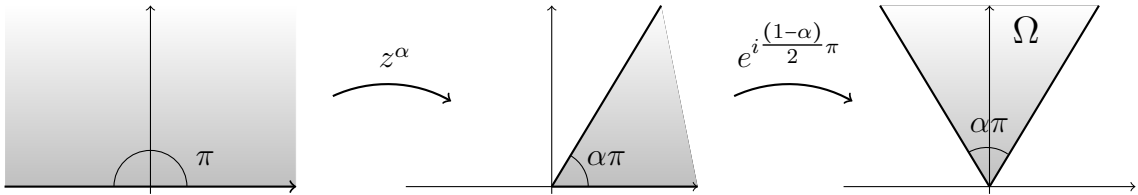
\begin{figure}[h!]
 \begin{center}
\begin{tikzpicture}[scale=.48]
\filldraw[draw=white,bottom color=lightgray, top color=white] (-24, 5)-- (-24,0)--(-16,0)--(-16,5) ; 
\filldraw[draw=white,bottom color=lightgray, top color=white] (-6, 5)-- (-9,0)--(-5,0) ; 
\filldraw[draw=white,bottom color=lightgray, top color=white] (0, 5) -- (3,0)--(6, 5) ; 
\draw [thick, ->] (-24,0) -- (-16,0); 
\draw [->] (-20,-.2) -- (-20,5); 
\draw (-19,0) arc (0:180:1);
\node[above] at (-18.5,0.3) {{$\pi $}};
\draw[thick,->] (-15,2.5) arc (115:60:100pt);
\node [above] at (-13.3,3) {{$z^\alpha$}};
\draw [->] (-13,0) -- (-5,0); 
\draw [->] (-9,-.2) -- (-9,5); 
\draw [ thick] (-9,0) -- (-6,5); 
\draw [ thick] (-9,0) -- (-5,0); 
\draw (-8,0) arc (0:60:1); 
\node[above] at (-7.5,0.3) {{$\alpha \pi $}}; 
\draw[thick,->] (-4,2.5) arc (115:60:100pt);
\node [above] at (-2.5,3) {{$e^{i \tfrac{(1-\alpha)}{2} \pi} $}};
\draw [ ->] (-1,0) -- (7,0); 
\draw [ ->] (3,-.2) -- (3,5);  
\draw [ thick] (6, 5) -- (3,0);  
\draw [ thick] (0, 5) -- (3,0); 
\draw (3.5,0.9) arc (55:120:1); 
\node[above] at (3,1.2) {{$\alpha \pi $}}; 
\node[below] at (4,5) {{\large$\Omega$}}; 
\end{tikzpicture}
\caption{Symmetric cone with aperture $\alpha \pi$.}
\label{fig1}
\end{center}
\end{figure}

We have the following result.

\begin{theorem} \label{thm:cone1}
Fix $\beta \in (-1,1)$. For any $f\in L^2(|x|^\beta)$ real-valued, it holds that
\begin{equation*}
\int_\R (Hf)^2 |x|^{\beta} \, dx = \int_\R f^2 |x|^{\beta}\, dx + 2 \cot\left( \frac{(1-\beta)\pi}{2}\right) \int_\R f Hf \sgn(x) |x|^{\beta}\, dx.
\end{equation*}
\end{theorem}

\begin{remark}
Note that $\cot\Big(\tfrac{(1-\beta)\pi}{2}\Big)$ blows up when $\beta=-1$ and $\beta=1$. This is consistent with the well-known fact that $|x|^\beta \in A_2$ if and only if $\beta\in (-1,1)$. 
\end{remark}

\begin{proof}
Fix $\beta \in (-1,1)$. Let $\alpha= 1-\beta \in (0,2)$, and consider the conformal map $\Phi$ given in \eqref{map1}.
We need to compute $\re$ and $\im$ on $\R$. If $x>0$, then $\Phi(x) =  i e^{-i \tfrac{\alpha}{2} \pi} x^\alpha.$
Differentiating with respect to $x$, we get
$$
\Phi'(x) = \alpha \sin\left(\frac{\alpha\pi}{2}\right) x^{\alpha-1} + i \alpha \cos\left(\frac{\alpha\pi}{2}\right) x^{\alpha-1}.
$$
Similarly, if $x<0$, then $\Phi(x) =  i e^{i \tfrac{\alpha}{2} \pi} (-x)^\alpha,$
and thus,
$$
\Phi'(x) = \alpha \sin\left(\frac{\alpha\pi}{2}\right) (-x)^{\alpha-1} - i \alpha \cos\left(\frac{\alpha\pi}{2}\right) (-x)^{\alpha-1}.
$$
Note that $|\Phi'(x)|^2 = \alpha^2 |x|^{2(\alpha-1)}.$ Therefore,
\begin{align*}
\re = \alpha^{-1} \sin\left(\frac{\alpha\pi}{2}\right)  |x|^{1-\alpha}\quad \hbox{and}\quad
\im = - \alpha^{-1} \cos\left(\frac{\alpha\pi}{2}\right) \sgn(x)  |x|^{1-\alpha}. 
\end{align*}
Substituting these expressions into \eqref{eq:rellich}, and using that $\beta=1-\alpha$, we obtain the result.
\end{proof}

\begin{corollary}
If $\beta\in (-1,1)$, it holds that
$$
\|H\|_{L^2(|x|^\beta)\to L^2(|x|^\beta)}^2  \ge 1 - \frac {2 (1+\beta)}\pi  \cot\left( \frac{(1-\beta)\pi}{2}\right) \int_1^\infty |x|^{-2-\beta}  \log  {|1-x|}\, dx.
$$
\end{corollary}

\begin{proof} By Theorem~\ref{thm:cone1}, with $f_r=\chi_{(0,r)}$ and $H(f_r)=\frac{1}{\pi}\log \frac{|x|}{|x-r|},$ we get that
\begin{eqnarray*} 
\int_\R (H(f_r))^2 |x|^{\beta} \, dx &=& \frac{r^{1+\beta}}{1+\beta}+ \frac 2\pi   \cot\left( \frac{(1-\beta)\pi}{2}\right)  \int_0^r   |x|^{\beta} \log \frac{|x|}{|x-r|} \, dx
\\
&=& \frac{r^{1+\beta}}{1+\beta}- \frac 2\pi   \cot\left( \frac{(1-\beta)\pi}{2}\right)  r^{1+\beta}  \int_1^\infty |y|^{-2-\beta} \log |1-y| \, dy
\\
&=& \|f_r\|_{L^2(|x|^\beta)}^2 \left( 1- \frac {2(1+\beta)}\pi   \cot\left( \frac{(1-\beta)\pi}{2}\right)  \int_1^\infty  |y|^{-2-\beta}  \log  {|1-y|} \, dy\right).
\end{eqnarray*}
Therefore,
$$
\|H\|_{L^2(|x|^\beta)\to L^2(|x|^\beta)}^2   \geq  \frac{ \| H(f_r)\|_{L^2(|x|^\beta)}^2 }{\|f_r\|_{L^2(|x|^\beta)}^2 } =  1- \frac {2(1+\beta)}\pi   \cot\left( \frac{(1-\beta)\pi}{2}\right)  \int_1^\infty  |x|^{-2-\beta}  \log  {|1-x|} \, dx.
$$
\end{proof}

\begin{remark}\label{re:sharp}
We observe that the right-hand side behaves as $ (1+\beta)^{-2}$ when $\beta \to -1^+$.
Furthermore, $|x|^\beta= e^{\beta \log |x|} = e^{ \beta \tfrac{\pi}{2} K(\sgn x)}= e^{f_1 + Kf_2}$, with $f_1 \equiv 0$ and $f_2= \tfrac{\beta \pi}{2}\sgn x$.  Hence,
$$[|x|^{\beta}]^2_{A_2(\hs)} \le \sec^2 (\tfrac{\beta\pi}{2})\sim(1+\beta)^{-2} \quad \hbox{as}~\beta\to-1^+.$$
Therefore, the dependence of $[w]_{A_2(\hs)}$ in Theorem \ref{thm:A2est} is sharp.
\end{remark}

\subsection{Monotone cones}\label{sec:moncones}

Let $\Omega$ be a monotone cone of aperture $\alpha \pi$ (see Figure~\ref{fig2}); then we must have $\alpha \in (1/2, 3/2)$.
Here the definition of $\Phi$ changes according to $\alpha \in (1/2,1)$ or $\alpha\in [1,3/2)$. We will only study the first case, since results from the second case can be deduced analogously.

Let $\alpha \in (1/2,1);$ for each  $\theta \in [1-\alpha,1/2]$, we define $\Phi : \R^2_+ \to \Omega$ such that
\begin{equation}\label{map2}
\Phi(z) = e^{i \theta \pi} z^\alpha = e^{i \theta \pi} e^{\alpha ( \log |z| + i \A(z) )},
\end{equation}
where we chose the branch cut $\{ i y : y \leq 0\}$. 

\bigskip

\medskip

\begin{figure}[h!]
 \begin{center}
\begin{tikzpicture}[scale=.48]
\filldraw[draw=white,bottom color=lightgray, top color=white] (-24, 5)-- (-24,0)--(-16,0)--(-16,5) ; 
\filldraw[draw=white,bottom color=lightgray, top color=white] (-12, 5)-- (-9,0)--(-5,0)--(-5,5) ; 
\filldraw[draw=white,bottom color=lightgray, top color=white] (0, 0) -- (5,1)--(6, 5)-- (0,5) ; 
\draw [thick, ->] (-24,0) -- (-16,0); 
\draw [->] (-20,-.2) -- (-20,5); 
\draw (-19,0) arc (0:180:1);
\node[above] at (-18.5,0.3) {{$\pi $}};
\draw[thick,->] (-15.3,2.5) arc (115:60:100pt);
\node [above] at (-13.5,3) {{$z^\alpha$}};
\draw [->] (-13,0) -- (-5,0); 
\draw [->] (-9,-.2) -- (-9,5); 
\draw [ thick] (-9,0) -- (-12,5); 
\draw [ thick] (-9,0) -- (-5,0); 
\draw (-8,0) arc (0:120:1); 
\node[above] at (-8,1) {{$\alpha \pi $}}; 
\draw[thick,->] (-4.3,2.5) arc (115:60:100pt);
\node [above] at (-2.8,3) {{$e^{i \theta \pi} $}};
\draw [ ->] (0,1) -- (8,1); 
\draw [ ->] (5,0) -- (5,5);  
\draw [ thick] (0, 0) -- (5,1);  
\draw [ thick] (5, 1) -- (6,5); 
\draw (5.15,1.7) arc (60:205:0.7);  
\node[above] at (3.7,1.5) {{$\alpha \pi $}};
\node[above] at (6.2,1.5) {{$\theta \pi $}};
\draw (6,1) arc (0:73:1);
\node[below] at (2,4.5) {{\large$\Omega$}}; 
\end{tikzpicture}
\caption{Monotone cone with aperture $\alpha \pi$.}
\label{fig2}
\end{center}
\end{figure}
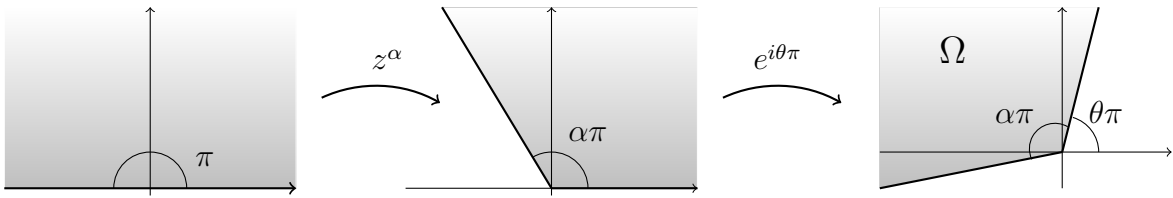

We next present several identities that follow from  \eqref{eq:rellich} and \eqref{eq:rellichbis}. Although it is possible that some of these identities may be known, we were unable to find them in the literature.

\begin{theorem}\label{thm:moncones}
Fix $\beta \in (0,1/2)$ and $\theta \in [\beta, 1/2]$. For any $f\in L^2(|x|^\beta)$ real-valued, it holds that
\begin{equation*}
\int_\R (Hf)^2 \,a_1 (\sgn x) |x|^{\beta} \, dx = \int_\R f^2 \,a_1(\sgn x)  |x|^{\beta}\, dx + 2 \int_\R f Hf \,a_2(\sgn x)  |x|^{\beta}\, dx,
\end{equation*} 
\begin{equation*}
\int_\R (Hf)^2 \,a_2 (\sgn x) |x|^{\beta} \, dx = \int_\R f^2 \,a_2(\sgn x)  |x|^{\beta}\, dx - 2 \int_\R f Hf \,a_1(\sgn x)  |x|^{\beta}\, dx,
\end{equation*} 
where $a_1$ and $a_2$ are the functions given by
\begin{align*}
a_1(s) &= \cos(\theta \pi) \frac{1+s}{2}  +  \cos((\theta-\beta)\pi) \frac{1-s}{2},  \\[0.2cm]
a_2 (s) &= \sin(\theta \pi) \frac{1+s}{2}  +  \sin((\theta-\beta)\pi) \frac{1-s}{2},
\end{align*}
for any $s$ in $\R$.
\end{theorem}

\begin{remark}
Note that if $\beta \in (0,1/2)$ and $\theta \in [\beta, 1/2]$,  then $0\leq \theta-\beta< 1/2$, and
$$
\min\big\{\cos(\theta\pi), \sin(\theta\pi), \cos((\theta-\beta)\pi), \sin((\theta-\beta)\pi) \big\} \geq 0.
$$
Therefore, $a_1(\sgn x) \geq 0$ and $ a_2(\sgn x) \geq 0$  for all $x\in \R.$
\end{remark}

\begin{proof}
Fix $\beta \in (0,1/2)$ and $\theta \in [\beta, 1/2]$. Let $\alpha= 1-\beta \in (1/2,1)$, and consider the conformal map given in \eqref{map2}. We proceed as in the previous proof.  If $x>0$, then $\Phi(x) =  e^{i\theta \pi} x^\alpha.$
Differentiating with respect to $x$, we get
$$
\Phi'(x) = \alpha  \cos(\theta\pi) x^{\alpha-1} + i \alpha \sin(\theta\pi) x^{\alpha-1}.
$$
Similarly, if $x<0$, then $\Phi(x) =  e^{i(\alpha+\theta) \pi} (-x)^\alpha.$ Therefore,
$$
\Phi'(x) = - \alpha \cos((\alpha+\theta)\pi)  (-x)^{\alpha-1}  - i \alpha  \sin((\alpha+\theta)\pi)  (-x)^{\alpha-1}.
$$
Since $|\Phi'(x)|^2 = \alpha^2 |x|^{2(\alpha-1)}$, and $\beta=1-\alpha$, it follows that
\begin{align*}
\re &=\alpha^{-1}\left(\cos(\theta \pi) \tfrac{1+\sgn x}{2}  - \cos((\alpha+\theta)\pi) \tfrac{1-\sgn x}{2}  \right) |x|^{1-\alpha} =(1- \beta)^{-1} a_1(\sgn x)|x|^\beta
\end{align*}
and
\begin{align*}
\im &=-\alpha^{-1} \left(\sin(\theta \pi) \tfrac{1+\sgn x}{2}  -  \sin((\alpha+\theta)\pi) \tfrac{1-\sgn x}{2} \right) |x|^{1-\alpha} = -(1-\beta)^{-1} a_2(\sgn x)|x|^\beta,
\end{align*}
where $a_1$ and $a_2$ are defined as in the statement. 
Substituting in \eqref{eq:rellich} and \eqref{eq:rellichbis} we conclude the desired results.
\end{proof}

When $\beta=\theta,$ Theorem~\ref{thm:moncones}  gives the following result:

\begin{corollary}\label{coro:moncones} If $\beta \in (0,1/2)$ and $f\in L^2(|x|^\beta)$ is real-valued, it holds that
\begin{align}
\cos(\beta\pi)\int_0^\infty& (Hf)^2\,|x|^\beta\,dx+ \int_{-\infty}^0 (Hf)^2|x|^\beta\,dx\nonumber\\
=&\cos(\beta\pi)\int_0^\infty f^2\,|x|^\beta\,dx+ \int_{-\infty}^0 f^2\,|x|^\beta\,dx+ 2\sin(\beta\pi) \int_0^\infty f Hf\,|x|^\beta\,dx,\label{eq:moncones1}
\end{align}
\begin{align}
\sin(\beta\pi)\int_0^\infty (Hf)^2\,|x|^\beta\,dx=&\sin(\beta\pi)\int_0^\infty f^2\,|x|^\beta\,dx\nonumber\\
&-2\cos(\beta\pi)\int_0^\infty f Hf\,|x|^\beta\,dx-2\int_{-\infty}^0 f Hf\,|x|^\beta\,dx.\label{eq:moncones2}
\end{align}

\end{corollary}

\medskip

Several interesting identities follow as particular cases of Corollary~\ref{coro:moncones}:

\medskip

 \underline{Identity~\eqref{eq:moncones1} and $\supp(f)\subset (0,\infty):$} If $\beta \in (0,1/2)$ and $f\in L^2(|x|^\beta)$ is real-valued, then
\begin{align*}
\cos(\beta\pi)\int_0^\infty& (Hf)^2\,|x|^\beta\,dx+ \int_{-\infty}^0 (Hf)^2|x|^\beta\,dx\\
=&\cos(\beta\pi)\int_0^\infty f^2\,|x|^\beta\,dx+ 2\sin(\beta\pi) \int_0^\infty f Hf\,|x|^\beta\,dx.
\end{align*}
As $\beta\to 1/2$ we obtain
\begin{align*}
 \int_{-\infty}^0 (Hf)^2|x|^{\frac{1}{2}}\,dx= 2 \int_0^\infty f Hf\,|x|^{\frac{1}{2}}\,dx.
\end{align*}

\bigskip

\underline{Identity~\eqref{eq:moncones1} and $\supp(f)\subset (-\infty,0):$} If $\beta \in (0,1/2)$ and $f\in L^2(|x|^\beta)$ is real-valued, then
\begin{align*}
\cos(\beta\pi)\int_0^\infty (Hf)^2\,|x|^\beta\,dx+ \int_{-\infty}^0 (Hf)^2|x|^\beta\,dx
= \int_{-\infty}^0 f^2\,|x|^\beta\,dx.
\end{align*}
As $\beta\to 1/2,$ we have
\begin{align*}
 \int_{-\infty}^0 (Hf)^2|x|^{\frac{1}{2}}\,dx= \int_{-\infty}^0 f^2\,|x|^{\frac{1}{2}}\,dx.
 \end{align*}

\bigskip

 \underline{Identity~\eqref{eq:moncones2} and $\supp(f)\subset (0,\infty):$} If $\beta \in (0,1/2)$ and $f\in L^2(|x|^\beta)$ is real-valued, then
\begin{align*}
\sin(\beta\pi)\int_0^\infty (Hf)^2\,|x|^\beta\,dx=\sin(\beta\pi)\int_0^\infty f^2\,|x|^\beta\,dx
-2\cos(\beta\pi)\int_0^\infty f Hf\,|x|^\beta\,dx.
\end{align*}
As $\beta\to 1/2,$ we get
\begin{align*}
\int_0^\infty (Hf)^2\,|x|^{\frac{1}{2}}\,dx=\int_0^\infty f^2\,|x|^{\frac{1}{2}}\,dx.
\end{align*}

\bigskip

 \underline{Identity~\eqref{eq:moncones2} and $\supp(f)\subset (-\infty,0)$}: If $\beta \in (0,1/2)$ and $f\in L^2(|x|^\beta)$ is real-valued, then
\begin{align*}
\sin(\beta\pi)\int_0^\infty (Hf)^2\,|x|^\beta\,dx=-2\int_{-\infty}^0 f Hf\,|x|^\beta\,dx.
\end{align*}
As $\beta\to 1/2,$ it follows that
\begin{align*}
\int_0^\infty (Hf)^2\,|x|^{\frac{1}{2}}\,dx=-2\int_{-\infty}^0 f Hf\,|x|^{\frac{1}{2}}\,dx.
\end{align*}

\section{Rellich identities for the Hilbert transform in $L^p$}\label{sec:rellichLp}
 
Theorem~\ref{thm:rellich} states Rellich's identities for the Hilbert transform for functions  in weighted $L^2$-spaces. In this section we present versions of such identities  for pairs of functions on  weighted Lebesgue spaces. Our main result is the following theorem.

 \begin{theorem}\label{thm:rellichLp} 
 Let  $\Phi$ be a conformal map as described in Section~\ref{sec:cm} such that $|\Phi'|\in A_p\cap A_{p'}$ for some $1<p<\infty.$ If $f\in L^p(|\Phi'|^{1-p})$ and  $g\in L^{p'}(|\Phi'|^{1-p'})$ are real-valued, then
 \begin{align}
 \int_\mathbb R \Big(  HfHg-fg  \Big)  \re dx &=  -\int_\mathbb R \Big( fHg+ gHf \Big) \im dx,\label{eq:rellichLp}\\
 \int_\mathbb R \Big(  HfHg-fg  \Big)  \im dx &=  \int_\mathbb R \Big( fHg+ gHf \Big) \re dx.\label{eq:rellichLpbis}
\end{align}
 \end{theorem}

\begin{proof} We first prove \eqref{eq:rellichLp} for $f$ and $g$ continuous with compact support. Using \eqref{eq:rellich} and the identity (see \cite[(5.1.23), p. 320]{MR3243734})
 \begin{equation}\label{eq:rellichLp1}
(Hf)^2 -f^2 = 2 H(f Hf), 
\end{equation}
we have that 
\begin{equation*}
\int_\mathbb R H( fHf)  \re dx = -\int_\mathbb R f Hf \im dx.
\end{equation*}
 Applying this equality to the functions $f+g$, $f$ and $g$, it follows that
$$
 \int_\mathbb R H\Big(   fHg+ gHf \Big)  \re dx = - \int_\mathbb R \Big( fHg+ gHf \Big) \im dx.
 $$
 Using \eqref{eq:rellichLp1} for the function $f+g$, $f$ and $g$,  we obtain
$$
HfHg-fg = H\Big(   fHg+ gHf \Big).
$$
These last two equalities give \eqref{eq:rellichLp}.

 The identity \eqref{eq:rellichLpbis} for $f$ and $g$ continuous with compact suppport is proved similarly using \eqref{eq:rellichbis}.
 
 The general case can be obtained by density once it is shown that the integrals in \eqref{eq:rellichLp} and \eqref{eq:rellichLpbis} are absolutely convergent for $f\in L^p(|\Phi'|^{1-p})$ and  $g\in L^{p'}(|\Phi'|^{1-p'}).$ For such functions,  using  H\"older's inequality, we obtain
\begin{equation*}
\int_\mathbb R   |f\,g| \, \re dx  \le   \left(\int_\mathbb R  |f|^p   |\Phi'|^{1-p} dx\right)^{\frac{1}{p}} \left(\int_\mathbb R  |g|^{p'}   |\Phi'|^{1-p'} dx\right)^{\frac{1}{p'}} <\infty.
\end{equation*}
We also have
\begin{align*}
\int_\mathbb R   |Hf\, Hg| \, \re dx
  &\le   \left(\int_\mathbb R  |Hf|^p   |\Phi'|^{1-p} dx\right)^{\frac{1}{p}} \left(\int_\mathbb R  |Hg|^{p'}   |\Phi'|^{1-p'} dx\right)^{\frac{1}{p'}} \\
&\lesssim   \left(\int_\mathbb R  |f|^p   |\Phi'|^{1-p} dx\right)^{\frac{1}{p}} \left(\int_\mathbb R  |g|^{p'}   |\Phi'|^{1-p'} dx\right)^{\frac{1}{p'}} <\infty,
\end{align*}
where in the last inequality we have used the boundedness of the Hilbert transform noting  that $|\Phi'|^{1-p} \in A_p$ since $|\Phi'|\in A_{p'}$ and  $ |\Phi'|^{1-p'} \in A_{p'}$ since $|\Phi'|\in A_p$. 
A similar reasoning is applied for the integrals on the right hand side of \eqref{eq:rellichLp} and \eqref{eq:rellichLpbis}.

  \end{proof}

 As an application, we next present examples of the identities of  Theorem~\ref{thm:rellichLp} associated to power weights. The next corollary follows by considering the conformal map $\Phi$ given in \eqref{map1} and the corresponding computations done in Section~\ref{sec:symcones}. 

 \begin{corollary} Let $1<p<\infty$ and $\beta\in (-1,1).$ If $f\in L^p(|x|^{\beta(p-1)})$ and $g\in L^{p'}(|x|^{\beta(p'-1)})$ are real-valued then
 \begin{align}
   \int_\mathbb R \Big(  HfHg-fg  \Big)  |x|^{\beta}  dx =  \cot\left(\frac{(1-\beta)\pi}{2}\right) \int_\mathbb R \Big( fHg+ gHf \Big) \sgn(x)  |x|^{\beta} dx,\\
 \int_\mathbb R \Big(  HfHg-fg  \Big)  \sgn(x)  |x|^{\beta}  dx =  - \tan\left(\frac{(1-\beta)\pi}{2}\right) \int_\mathbb R \Big( fHg+ gHf \Big)  |x|^{\beta} dx.
 \end{align}
 \end{corollary}
 
 \bigskip
 
 \subsection*{Acknowledgements}
We thank the referee for their careful reading of the manuscript and their suggestions.
 \bibliographystyle{plain}
\bibliography{CNSbiblio}

 \end{document}